\documentclass{amsart}

\usepackage[margin=1.5in]{geometry} 
\usepackage{amsmath,amsthm,amssymb,hyperref,mathrsfs,mathtools}
\usepackage[colorinlistoftodos]{todonotes}
\usepackage{multicol}
\usepackage{enumitem}
\usepackage{ulem} 
\usepackage{color,soul}
\usepackage{subcaption}
\newcommand{\R}{\mathbb{R}}  
\newcommand{\Z}{\mathbb{Z}}
\newcommand{\N}{\mathbb{N}}

\newcommand{\ol}[1]{\overline{#1}}

\newcommand{\la}{\lambda}

\renewcommand{\angle}{\measuredangle}

\newtheorem{ppn}{Proposition}
\newtheorem{cor}{Corollary}
\newtheorem{thm}{Theorem}
\newtheorem*{rmk*}{Remark}
\newtheorem{lma}{Lemma}
\newtheorem*{plm}{Problem}

\AtEndDocument{%
  \par
  \medskip
  \begin{tabular}{@{}l@{}}%
    \textsc{Department of Mathematics, Weizmann Institute of Science, Rehovot 76100 Israel}\\
    \textit{E-mail address}: \texttt{rotem.assouline@weizmann.ac.il}
  \end{tabular}}
\date{}

\title{Growth competitions on spherically symmetric Riemannian manifolds}
\author{Rotem Assouline}
\begin{document}

\normalem
\maketitle
\begin{abstract}
    We propose a model for a growth competition between two subsets of a Riemannian manifold. The sets grow at two different rates, avoiding each other. It is shown that if the competition takes place on a surface which is rotationally symmetric about the starting point of the slower set, then if the surface is conformally equivalent to the Euclidean plane, the slower set remains in a bounded region, while if the surface is nonpositively curved and conformally equivalent to the hyperbolic plane, both sets may keep growing indefinitely.

\end{abstract}

\section{Introduction}

Consider two subsets $A,B$ of, say, the Euclidean plane, which evolve over time, $A=A_t,B=B_t$, according to the following simple rules: both sets begin as  singletons, $A_0=\{q\},B_0=\{p\}$ for some points $p\ne q$, and expand at rates $\la>1$ and 1 respectively, without intersecting each other. If a point belongs to one of the sets at a certain time, then it remains in the set forever.  What will the sets $A_\infty:=\bigcup_tA_t$ and $B_\infty:=\bigcup_tB_t$ look like?
    
This \textit{growth competition} can take place on essentially any metric space. Its precise formulation is given in the following section, and its existence and uniquness on complete Riemannian manifolds is established. We then study growth competitions on Riemannian manifolds which are spherically symmetric about the point $p$. The cases of the Euclidean plane and the hyperbolic plane exhibit contrasting behaviors: on the Euclidean plane, the faster set $A$ will eventually trap the set $B$ in a bounded region, while on the hyperbolic plane, if the two sets begin sufficiently far apart, then there is \textit{coexistence}, i.e., both sets keep expanding forever. In fact, we show:
\begin{thm}\label{mainthm}
    Let $M$ be a two-dimensional, complete, non-compact Riemannian manifold, which is rotationally symmetric about $p\in M$. Let $q\in M$ and $\la>1$, and let $\{A_t,B_t\}$ be the corresponding growth competition. Then
    \begin{enumerate}
        \item If $M$ is parabolic, then $B_\infty$ is bounded.
        \item If $M$ is hyperbolic and nonpositively curved, then  $B_\infty$ is unbounded if $d(p,q)$ is sufficiently large.
    \end{enumerate}
\end{thm}
 
Here \textit{parabolic} (resp. \textit{hyperbolic}) stands for ``conformally equivalent to the Euclidean (resp. hyperbolic) plane". 
    
The problem of determining the shapes of the competing sets was suggested by Itai Benjamini \cite{itaisurvival}, and is loosely inspired by probabilistic competitions on $\Z^d$ and other graphs, see  \cite{deijfen2015pleasures}, \cite{haggstrom_pemantle_1998}. A related problem, concerning a strategy to control a forest fire on the Euclidean plane, was introduced by Bressan, see \cite{BRESSAN2007179}, \cite{BressanDelellis}, \cite{BressanWang}. In \cite{Hamobs}, Bressan's fire confinement problem is treated using the apparatus of viscosity solutions to the Hamilton-Jacobi equation. This could be the appropriate framework for dealing with problems such as the one introduced here. However, as we are interested mostly in the relationship between coexistence and the underlying geometry, we have managed to define the competition in a way which enables an elementary proof of existence and uniqueness of the solution, while still capturing the essence of the problem.
    
Given the connection presented here between the conformal type of a simply-connected surface, and the outcome of growth competitions on it, it is now natural to ask:
\begin{plm}
    Prove or disprove each of the following statements.
    \begin{enumerate}
        \item On a hyperbolic surface, for every $\la>1$ there is some choice of $p,q$ such that $B_\infty$ is unbounded.
        \item On a parabolic surface, for every $\la>1$ and every $L>0$, there is some choice of $p,q$ such that $d(p,q)>L$ and $B_\infty$ is bounded.
        \end{enumerate}
\end{plm}
    
\textbf{Acknowledgements:} I would like to express my gratitude to Itai Benjamini for offering me this problem and assisting in writing the paper, and to Bo'az Klartag for pointing me to the paper \cite{milnor}, as well as reviewing the proof thoroughly and suggesting improvements.
    
Partially supported by a grant from the Israel Science Foundation (ISF).
    
\begin{figure}[t]
    \centering
    \begin{subfigure}{.53\textwidth}
        \centering
        \includegraphics[width=.95\textwidth]{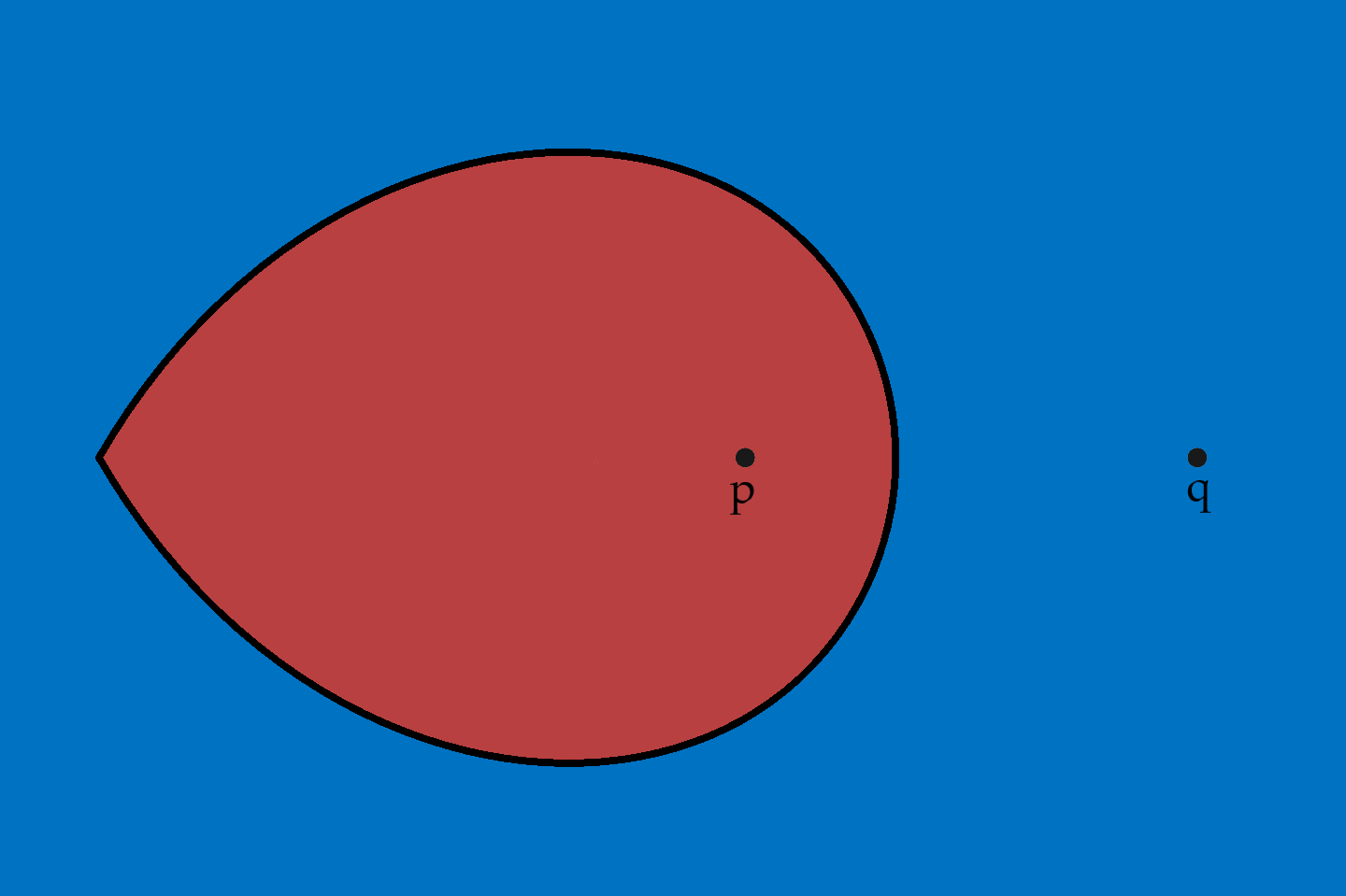} 
        \caption{}
    \end{subfigure}\hfill
    \begin{subfigure}{0.47\textwidth}
        \centering
        \includegraphics[width=.95\textwidth]{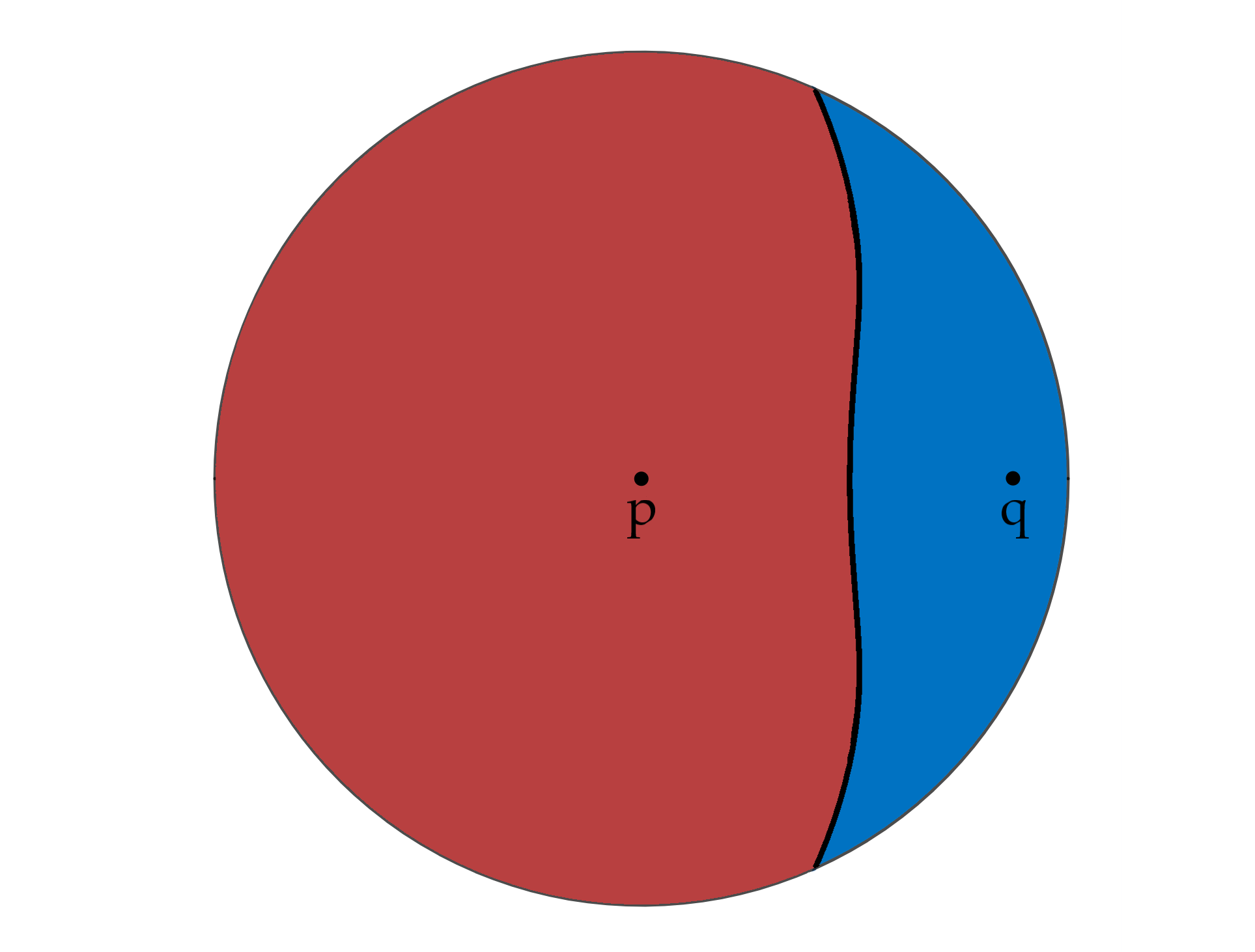} 
        \caption{}
    \end{subfigure}
    \caption{The sets $A_\infty$ (blue) and  $B_\infty$ (red) for growth competitions     on (a) the Euclidean plane (b) the hyperbolic plane.}
    \label{ABfig}
\end{figure}

\section{Growth competitions}\label{growths}
Let $X$ be a metric space, let $p\ne q\in X$, and let $\lambda>1$. Let $\{A_t,B_t\}_{t\ge 0}$ be two families of subsets of $X$. Say that a path $\gamma:[0,T]\to X$ \textit{avoids $A$} (resp. \textit{avoids $B$}) if
\begin{itemize}
    \item  $\gamma(0)=p$ (resp. $\gamma(0)=q$),
    \item $\gamma$ is 1-Lipschitz (resp. $\la$-Lipschitz)
    \item $\gamma(s)\notin A_s$ for all $s\in[0,T]$ (resp. $\gamma(s)\notin B_s$ for all $s\in[0,T]$).
\end{itemize}
The pair $A_t,B_t$ will be called a \textit{ growth competition between $p$ and $q$}, if 
\begin{itemize}
    \item $A_t,B_t$ are increasing in $t$, i.e., $A_{t_1}\subseteq A_{t_2}$ and $B_{t_1}\subseteq B_{t_2}$ for all $0\le t_1\le t_2$, 
    \item $A_0=\{q\}, B_0=\{p\}$, and for all $t> 0$:
    \begin{equation}\label{ABconditions}
        A_t=\bigcup_\gamma\gamma([0,t]) \qquad \text{and} \qquad B_t=\bigcup_\eta\eta([0,t))
    \end{equation}
    where the union on the left is over paths avoiding $B$, and the union on the right is over paths avoiding $A$.
\end{itemize}
Denote also
\begin{equation}
    A_\infty=\bigcup_{t\ge 0}A_t \qquad \text{and} \qquad B_\infty=\bigcup_{t\ge 0}B_t.
\end{equation} 
In the rest of this section, we prove existence and uniqueness of growth competitions on complete Riemannian manifolds. Fix a manifold $(M,g)$, two points $p\ne q\in M$, and $\la>1$.

The following lemma is quite evident from the definitions.
\begin{lma}\label{ABdisjoint}
    Let $\{A_t,B_t\}$ be a growth competition between $p$ and $q$. Then $A_t\cap B_{t'}=\varnothing$ for all $t,t'\in[0,\infty]$.
\end{lma}
\begin{proof}
    Let $0< t<\infty$ and let $x\in A_t$. There exists a path $\gamma:[0,t]\to M$ avoiding $B$ such that $x=\gamma(t_0)$ for some $0\le t_0\le t$, and since $\gamma\vert_{[0,t_0]}$ is also a curve avoiding $B$, $x\in A_{t_0}$. Let $\eta:[0,t')\to M$ be a curve avoiding $A$. We claim that $\eta(s)\ne x$ for all $s\in[0,t')$. Indeed,
    \begin{itemize}
        \item[--] If $s<t_0$ then $\eta(s)\in B_{t_0}$, so $\eta(s) \ne x = \gamma(t_0)$ because $\gamma$ avoids $B$.
        \item[--] If $s\ge t_0$ then $x = \gamma(t_0)\in A_{t_0}\subseteq A_s$, so $\eta(s) \ne x$ since $\eta$ avoids $A$. 
    \end{itemize}
    Thus $\eta(s)\ne x$ for all $s\in[0,t')$, and since $\eta$ is an arbitrary curve avoiding $A$, $x\notin B_{t'}$. This proves that $A_t,B_{t'}$ are disjoint for all $t,t'\ge 0$, and it follows that $A_\infty$ and $B_\infty$ are disjoint.
\end{proof}

Denote by $d$ the distance function of $(M,g)$. For $x\in M$ and $R>0$, denote by $\mathcal{B}(x,R)=\{y\in M \mid d(x,y)<R\}$ the open ball of radius $R$ centered at $x$. Another fact which follows trivially from the definitions is the following.
\begin{lma}
    $A_t\subseteq \overline{\mathcal{B}(q,\la t)}$ and $B_t\subseteq \mathcal{B}(p,t)$ for all $t\ge  0$.
\end{lma}
We shall now construct a certain subset of $M$ and prove that it must coincide with $B_\infty$. This will leave us with a unique candidate for a growth competition.\\
For a subset $S\subseteq M$, denote by $d_S$ the \textit{intrinsic metric of $M\setminus S$}, which is the metric on $M\setminus S$ defined by
\begin{equation*}
    d_S(x,y)=\inf\{\mathrm{Length}(\gamma) \mid \gamma:[0,T]\to M\setminus S \text{ is a Lipschitz path}, \  \gamma(0)=x, \gamma(T)=y\}.
\end{equation*}
If the set $S$ is open, then by the Arzelà–Ascoli theorem, the metric space $(M\setminus S,d_S)$ is a geodesic metric space, i.e., for each $x,y\in M\setminus S$ there exists a path $\gamma$ realizing their intrinsic distance. If $M\setminus S$ is not path-connected, then $d_S$ might attain the value $\infty$. 
    
Define a sequence $\Omega_n\subseteq M$ for $n\ge 0$ by recursion as follows. Set $\Omega_0:=\varnothing$. Having defined $\Omega_0,\dots,\Omega_n$, set
\begin{equation}
    \Omega_{n+1}=\Omega_n\cup\{x\in M\setminus\Omega_n \mid d_{\Omega_n}(x,q)> \la d(x,p)\}.
    \end{equation}
In particular,
\begin{equation}\label{omega1}
    \Omega_1=\{x\in M \mid d(x,q)>\la d(x,p)\}.
\end{equation}
Finally, set 
\begin{equation*}
    \Omega:=\bigcup_{n=1}^\infty\Omega_n.
\end{equation*}

\begin{lma}
    The sets $\Omega_n$ are open.
\end{lma}
\begin{proof}
    $\Omega_0$ is trivially open. Assume that $\Omega_n$ is open. Then the distance function $d_{\Omega_n}(\cdot,q)$ is lower semicontinuous on $M \setminus \Omega_n$ with respect to the metric $d$, so the set $M \setminus \Omega_{n+1}$ is closed in $M \setminus \Omega_n$, and therefore in $M$. Thus $\Omega_{n+1}$ is open.
\end{proof}
\begin{lma}\label{eqtyineqty}
    Let $n\ge 0$. Then 
    \begin{align}
        d_{\Omega_{n}}(x,q)&\ge\la d(x,p) \qquad \text{ for all $x\in\partial\Omega_n$, and}\label{ineqty1}\\
        d_{\Omega_{n}}(x,q)&=\la d(x,p) \qquad \text{ for all $x\in\partial\Omega_{n+1}$.}\label{eqty1}
    \end{align}
\end{lma}
\begin{proof}
    For $n=0$, \eqref{ineqty1} is vacuous and \eqref{eqty1} is obvious from \eqref{omega1}. Now let $n\ge 1$ and assume that the claim holds for $n-1$. Let $x\in\partial\Omega_n$. Then since $\Omega_{n-1}\subseteq\Omega_n$, we have
    $$d_{\Omega_n}(x,q)\ge d_{\Omega_{n-1}}(x,q)=\la d(x,p).$$
    Now let $x\in\partial\Omega_{n+1}$. Then in particular $x\in M\setminus\Omega_{n+1}$, which by definition means $d_{\Omega_n}(x,q)\le\la d(x,p)$. Now, either $x\in\partial\Omega_n$, in which case reverse inequality holds by \eqref{ineqty1}, or $x\notin\ol{\Omega_n}$, and then the reverse inequality holds since $d_{\Omega_n}$ is continuous on $M\setminus\overline{\Omega_n}$ and $x\in\partial\Omega_{n+1}$.
\end{proof}

A subset $S\subseteq M$ is said to be \textit{star-shaped} about a point $x_0\in S$ if for every $x\in S$, any minimizing geodesic joining $x_0$ and $x$ lies inside $S$.
\begin{lma}
    The set $\Omega_n$ is star-shaped about $p$ for all $n\ge0 $. Thus $\Omega$ is star-shaped about $p$.
\end{lma}
\begin{proof}
    For $n=0$ there is nothing to prove. Let $n\ge1$ and let $x\in\Omega_n$. We must show that a minimizing geodesic joining $p$ and $x$ lies in $\Omega_n$. By induction we may assume that $x\notin\Omega_{n-1}$, so that $d_{\Omega_{n-1}}(x,q)>\la d(x,p)$. Let $\gamma$ be a unit-speed minimizing geodesic from $p$ to $x$. By induction, there exists some $t_0>0$ such that $\gamma(t)\in\Omega_{n-1}$ exactly when $t<t_0$, and $x=\gamma(t_1)$ for some $t_1\ge t_0$. Since $\gamma\vert_{[t_0,t_1]}$ lies outside $\Omega_{n-1}$, we have that $d_{\Omega_{n-1}}(\gamma(t),x)=t_1-t$ for all $t_0\le t\le t_1$, and so
    \begin{equation*}
        d_{\Omega_{n-1}}(\gamma(t),q)\ge d_{\Omega_{n-1}}(x,q)-d_{\Omega_{n-1}}(x,\gamma(t))>\la d(x,p)-(t_1-t)=(\la-1)t_1+t\ge \la t
    \end{equation*}
    for all $t_0\le t\le t_1$, which implies that $\gamma(t)\in\Omega_n$ as desired.
\end{proof}

The following lemma states the key property of $\Omega$.
\begin{lma}\label{Omegaprop0}
    For all $x\in M\setminus\Omega$,
    \begin{equation}
        d_\Omega(x,q)\le\la d(x,p),
    \end{equation}
    with equality if and only if $x\in\partial \Omega$. 
\end{lma}
\begin{proof}
    Let $x\in\partial\Omega$. For each $n\ge 1$, let $x_n\in\ol{\Omega_n}$ be a point of minimal distance to $x$. Clearly $x_n\to x$ as $n\to\infty$. The point $x_n$ must lie in $\partial\Omega_n$, for otherwise we would have $x\in\Omega_n\subseteq\Omega$. \\
    Now, on one hand, by \eqref{eqty1},
    \begin{align}\label{dOmegacomp}
        \la d(x_n,p) & = d_{\Omega_{n-1}}(x_n,q)\le d_{\Omega_{n-1}}(x_n,x)+d_{\Omega_{n-1}}(x,q) \le d_{\Omega_n}(x_n,x)+d_\Omega(x,q),
    \end{align}
    and taking $n\to \infty$ we get $\la d(x,p)\le d_\Omega(x,q)$. On the other hand, again by \eqref{eqty1}, there are curves $\gamma_n$ joining $q$ and $x_n$ of length $\la d(x_n,p)$ and lying outside $\Omega_{n-1}$, and by the Arzel{\'a}-Ascoli theorem, a subsequence of them converges to a curve $\gamma$ lying outside $\Omega_{n-1}$ for all $n$, and thus outside $\Omega$, and of length $\la d(x,p)$. This implies that $d_\Omega(x,q)\le \la d(x,p)$, and together with \eqref{dOmegacomp} we have equality.\\
    Now let $x\in M\setminus\ol{\Omega}$, and let $\gamma$ by a unit-speed minimizing geodesic from $p$ to $x$. Since $\Omega$ is star-shaped, we  there exist $0<t_0<t_1$ such that $\gamma(t_1)=x$ and $\gamma(t)\in\Omega$ exactly when $t<t_0$. Let $x'=\gamma(t_0)$. Then $x'\in\partial\Omega$, and since $\gamma\vert_{[t_0,t_1]}$ lies outside $\Omega$, we get
    \begin{align*}
        d(x,p)&=d(x,x')+d(x',p)=d_{\Omega}(x,x')+\la^{-1}d_\Omega(x',q)> \la^{-1}(d_\Omega(x,x')+d_\Omega(x',q))\\&\ge \la^{-1}d_\Omega(x,q).
    \end{align*}
\end{proof}

We now have what we need in order to argue that $\Omega$ coincides with $B_\infty$, and write down the solution to the competition.
\begin{ppn}\label{OmegaisB}
    Let $\la>1$. Let $\{A_t,B_t\}$ be a growth competition between $p,q$. Then $B_\infty=\Omega$ and $A_\infty=M\setminus\Omega$.
\end{ppn}
\begin{proof}
    For the inclusion $\Omega\subseteq B_\infty$, we must show that $\Omega_n\subseteq B_\infty$ for all $n\in\N$.\\
    For $n=0$ there is nothing to prove. Let $n\ge 1$. Let $\gamma$ be a unit speed geodesic emanating from $p$, and suppose that $\gamma$ does not avoid $A$. Let $t_0=\inf\{t \mid \gamma(t)\in A_t\}$. There is a sequence of curves $\eta_k$ avoiding $B$ such that $\eta_k(s_k)=\gamma(s_k)$ for some $s_k\searrow t_0$. By induction, $\eta_k\vert_{[0,s_k]}$ lie outside $\Omega_{n-1}$, so $d_{\Omega_{n-1}}(\gamma(s_k),q)\le \mathrm{Length}(\eta_k\vert_{[0,s_k]})\le \la s_k$ since $\eta_k$ are $\la$-Lipschitz. Taking $k\to \infty$ we see that $\gamma(t_0)\notin \Omega_{n-1}$ and  $d_{\Omega_{n-1}}(\gamma(t_0),q)\le \la d(\gamma(t_0),p)$, whence $\gamma(t_0)\notin\Omega_n$. Thus any unit speed geodesic emanating from $p$ and staying inside $\Omega_n$ is a curve avoiding $A$. Since $\Omega_n$ is star-shaped about $p$, this proves that $\Omega_n\subseteq B_\infty$ for all $n\in \N$ whence $\Omega\subseteq B_\infty$.
    
    In the other direction, let $\gamma$ be a 1-Lipschitz path with $\gamma(0)=p$, and suppose that $\gamma(t_1)\notin\Omega$ for some $t_1$. Then by Lemma \ref{Omegaprop0}, there is a path $\eta$ of length at most $\la d(\gamma(t_1),p)\le\la t_1$ joining $q$ and $\gamma(t_1)$ and lying entirely outside $\Omega$, and if we take $\eta$ to be minimizing and parametrized by speed $\la$ then $\eta(t_2)=\gamma(t_1)$ for some $t_2\le t_1$, and $d(\eta(t),p)\ge \la^{-1}d_{\Omega}(\eta(t),q)=t$ for all $t\in[0,t_1]$ by Lemma \ref{Omegaprop0}, whence $\eta$ is a curve avoiding $B$ and therefore $\gamma(t_1)=\eta(t_2)\in A_{t_2}\subseteq A_{t_1}$. It follows that $\gamma$ does not avoid $A$. Thus curves avoiding $A$ cannot leave $\Omega$ and therefore $B_\infty\subseteq \Omega$.
    
    We have shown that $B_\infty=\Omega$. Lemma \ref{Omegaprop0} implies that $M\setminus\Omega$ is path-connected, so for every $x\in M\setminus\Omega$, there is a curve $\eta$ lying outside $\Omega$ and joining $q$ to $x$. Since $B_\infty\subseteq\Omega$, the curve $\eta$ avoids $B$ and therefore $x\in A_\infty$. On the other hand, by Lemma \ref{ABdisjoint}, $A_\infty\subseteq M\setminus \Omega$.
\end{proof}

\begin{cor}\label{solcor}
    For each $p,q\in M$ and $\la>1$, there is a unique growth competition, given by
    \begin{equation}\label{sol}
    \begin{split}
        A_t & = \overline{\mathcal{B}}_\Omega(q,\lambda t)\\
        B_t & =\Omega \cap \mathcal{B}(p,t),
    \end{split}
    \end{equation}
    where $\overline{\mathcal{B}}_\Omega(x,R)$ denotes the closed ball of radius $R$ in the metric space $(M\setminus \Omega,d_\Omega)$.
\end{cor}
\begin{proof}
    Proposition \ref{OmegaisB}, Lemma \ref{OmegaisB} and the fact that $\Omega$ is star-shaped about $p$, imply that any growth competition must take the form \eqref{sol}. So it remains to show that this is indeed a growth competition. Clearly any path starting at $p$ and staying inside $\Omega$ avoids $A$, and any path starting at $q$ and not intersecting $\Omega$ avoids $B$. So we have the inclusions  $\subseteq$ in \eqref{ABconditions}. 
    
    For the inclusions $\supseteq$, first let $\gamma$ be a curve avoiding $A$. The same argument as in the proof of Proposition \ref{OmegaisB} implies that $\gamma$ remains inside $\Omega$, and since it is 1-Lipschitz, it is contained in $\mathcal{B}(p,t)$. Now let $\gamma:[0,t]\to M$ be a curve avoiding $B$; we argue by induction that it does not enter $\Omega_n$. Again $n=0$ is trivial. Let $n \ge 0$ and assume that $\gamma(t_0) \in \Omega_{n+1}$. By the definition of $\Omega_{n+1}$, and since $\gamma$ does not intersect $\Omega_n$ by induction, $\mathrm{Length}(\gamma\vert_{[0,t_0]}) > \lambda d(x,p)$. Since $\gamma$ is $\lambda$-Lipschitz, it follows that $t_0>d(x,p)$, whence $\gamma(t_0) \in B_{t_0}$, which is a contradiction to the assumption that $\gamma$ avoids $B$. Thus $\gamma(s) \notin \Omega_n$ for all $s\in[0,t]$ and all $n\in\N$, whence $\gamma(s)\notin\Omega$ for all $s\in[0,t]$ . Since $\gamma$ is $\lambda$-Lispchitz, $\gamma(s)\in\overline{\mathcal{B}}_\Omega(q,\lambda t)$ for all $s\in[0,t]$. This finishes the proof of the inclusions $\supseteq$ in \eqref{ABconditions}.
\end{proof}

\section{Coexistence on spherically symmetric manifolds}\label{rotsym}

We now restrict our attention to growth competitions taking place on complete, non-compact manifolds which are spherically symmetric about the point $p$. By spherically symmetric we mean that the metric takes the form
\begin{equation}\label{rotsymmetric}
    g=dr^2+G(r)^2 d\theta^2
\end{equation}
where $(r,\theta)\in(0,R)\times S^{n-1}$ are polar normal coordinates centered at $p$, $d\theta^2$ is the standard metric on $S_pM\cong S^{n-1}$, where $S_xM$ denotes the unit sphere of the tangent space to $M$ at $x$, and $G$ is a smooth positive function satisfying $G\to 0$ and $G/r\to 1$ as $r\to 0$. Since $M$ is noncompact, the coordinates $(r,\theta)$ are global, i.e. $R=\infty$, and $M$ is diffeomorphic to $\R^n$.
Set
\begin{equation*}
    \ell:=d(p,q).
\end{equation*}
Let $u_0\in S_pM$ denote the initial velocity vector of the unit-speed geodesic from $p$ to $q$, so that in polar coordinates, $q=(\ell,u_0)$.

Coexistence in this setting is related to the convergence of the integral
\begin{equation}\label{Gint}
    I:=\int_1^\infty\frac{1}{G(r)}dr.
\end{equation}
\begin{ppn}\label{bbdd}
    Suppose that $M$ is spherically symmetric about $p$, and that $I=\infty$. Then $B_\infty$ is bounded.
\end{ppn}
\begin{rmk*}
    In particular, if $M = \R^n$ then $B_\infty $ is bounded.
\end{rmk*}
\begin{proof}
    We prove that on every geodesic ray emanating from $p$, there is a point which can be reached by a path avoiding $B$ within some time $T$ independent of $\theta$. Since $B_\infty$ is star-shaped, the proposition will follow because then $B_\infty\subseteq \mathcal{B}(p, T)$.
    
    Let $\theta\in S_pM$ and let $\eta$ be a unit-speed great circle in $S_pM$ with $\eta(0)=u_0$ and $\eta(\tau)=\theta$ for some $\tau\le\pi$. Define a path $\gamma:[0,T]\to M$ in polar coordinates by 
    \begin{equation*}
        \gamma(t)=(\ell+t,\eta(\alpha(t)))
    \end{equation*}
    where
    \begin{equation*}
        \alpha(t):=\sqrt{\la^2-1}\int_{0}^{t}G(\ell+s)^{-1}ds.
    \end{equation*}
    Since $\alpha$ tends to $\infty$ with $t$, we can choose $T$ such that $\alpha(T)=\tau$, and therefore $\gamma(T)=(\ell+T,\eta(\tau))=(\ell+t,\theta)$ lies on the ray from $p$ with direction $\theta$. The path $\gamma$ avoids $B$; indeed, $\gamma(0)=(\ell,u_0)=q$, $\gamma$ is $\la$-Lipschitz because
    \begin{equation*}
        |\dot\gamma(t)|^2=|(1,\alpha'(t)\dot\eta(\alpha(t)))|^2=1+G(\ell+t)^2(\la^2-1)G(\ell+t)^{-2}|\dot\eta(\alpha(t))|^2=\la^2,
    \end{equation*}
    and $d(\gamma(t),p)=\ell+t>t$ for all $t\in[0,T]$, so $\gamma(t)\in M\setminus \mathcal{B}(p,t)\subseteq M\setminus B_t$.
\end{proof}

The condition $I=\infty$ is not necessary. First, observe that if $\ell$ is small, then the competition resembles a Euclidean competition and we cannot expect coexistence, no matter what $G$ is. Second, having $I<\infty$ does not prevent $M$ from containing spheres centered around $p$ with arbitrarily large radius yet arbitrarily small surface area, and if $q$ lies on such a sphere then the set $A$ will conquer the entire sphere quickly, trapping $B$ within the ball it bounds. Thus, at least as long as $p$ remains at the origin, some extra assumption is needed in order to enable coexistence. Motivated by Itai Benjamini's observation \cite{itaisurvival} that coexistence is possible on Gromov-hyperbolic spaces, we add the assumption that $M$ is nonpositively curved, i.e. that all its sectional curvatures are nonpositive.
\begin{ppn}\label{unbounded}
    Suppose that $M$ is spherically symmetric about $p$ and nonpositively curved, and that $I<\infty$. Then there exists $L>0$ such that if $\ell>L$ then $B_\infty$ is unbounded.
\end{ppn}
\begin{rmk*}\normalfont
    Note that if the sectional curvature of $M$ is bounded from above by a negative constant then automatically $I<\infty$. In this case one can use the argument from  \cite{itaisurvival} to prove that coexistence is possible.
\end{rmk*}

Recall that $B_\infty=\Omega$ by Proposition \ref{OmegaisB}. Since $M$ is nonpositively curved, any two points are joined by a unique minimizing geodesic. Say that a point $x\in M$ is \textit{visible} if the unique minimizing geodesic joining $q$ and $x$ does not intersect $\Omega$, and that $x$ is \textit{visible$_n$} if this geodesic does not intersect $\Omega_n$.

Since $\Omega$ is open and star-shaped, there is a function $f:S_pM\to(0,\infty]$ such that $\Omega$ is given in polar normal coordinates by the relation
$$\Omega=\{(r,\theta) \mid r<f(\theta)\}$$
(note that $f$ may attain the value $\infty$). 
Similarly, let $f_1$ denote the function corresponding to the star-shaped set $\Omega_1$.

\begin{lma}\label{omega0geolma}
    Suppose that $M$ is spherically symmetric about $p$ and nonpositively curved. A geodesic from $q$ meets $\partial\Omega_1$ at most twice. If it meets $\partial\Omega_1$ at $x_1=(r_1,\theta_1)$ and $x_2=(r_2,\theta_2)$, and $\angle(\theta_1,u_0)<\angle(\theta_2,u_0)$, then $x_1$ is visible$_1$, while $x_2$ is not.
\end{lma}
Here $\angle(v,w)=\arccos\left<v,w\right>$ for $v,w\in S_xM$, $x\in M$.
\begin{proof}
    Let $\gamma$ be a unit-speed geodesic emanating from $q$. The function $$h(t):=\la d(\gamma(t),p)-d(\gamma(t),q)=\la d(\gamma(t),p)-1$$ 
    is strictly convex since $M$ is nonpositively curved, unless $\gamma$ passes through $p$, in which case $h(t)=\la |\ell-t|-1$. In both cases, $h$ vanishes at most twice, and is negative between its zeros; hence $\gamma^{-1}(\Omega_1)$ is an interval (possibly empty) and $\gamma^{-1}(\partial\Omega_1)$ consists of at most two points; if it consists of exactly two points $t_1<t_2$, then only $\gamma(t_1)$ is visible$_1$, since $\gamma\vert_{[0,t_2]}$ intersects $\Omega_1$ and $\gamma\vert_{[0,t_1]}$ does not. Let
    $$\gamma(t)=(\rho(t),\eta(t))$$
    be the expression of $\gamma$ in polar coordinates. Since $M$ is spherically symmetric, the curve $\eta$ lies on a great circle in $S_pM$, and $\angle(\eta(t),\eta(0))=\angle(\eta(t),u_0)$ is strictly increasing. Therefore, if $x_1,x_2$ are as in the statement of the lemma, and $t_i$ satisfy $\gamma(t_i)=x_i$,  then $t_1<t_2$, hence $x_1$ is visible$_1$ while $x_2$ is not.
\end{proof}

\begin{lma}\label{visibilitylma}
    Suppose that $M$ is spherically symmetric about $p$ and nonpositively curved. Then there exists $0<a<\pi$ such that $(f_1(\theta),\theta)\in\partial\Omega_1$ is visible$_1$ exactly when $\angle(\theta,u_0)\le a$. Moreover, $a$ is bounded away from $\pi$ in terms of $\la$ alone.
\end{lma}
\begin{figure}
    \centering
    \includegraphics[width=.6\textwidth]{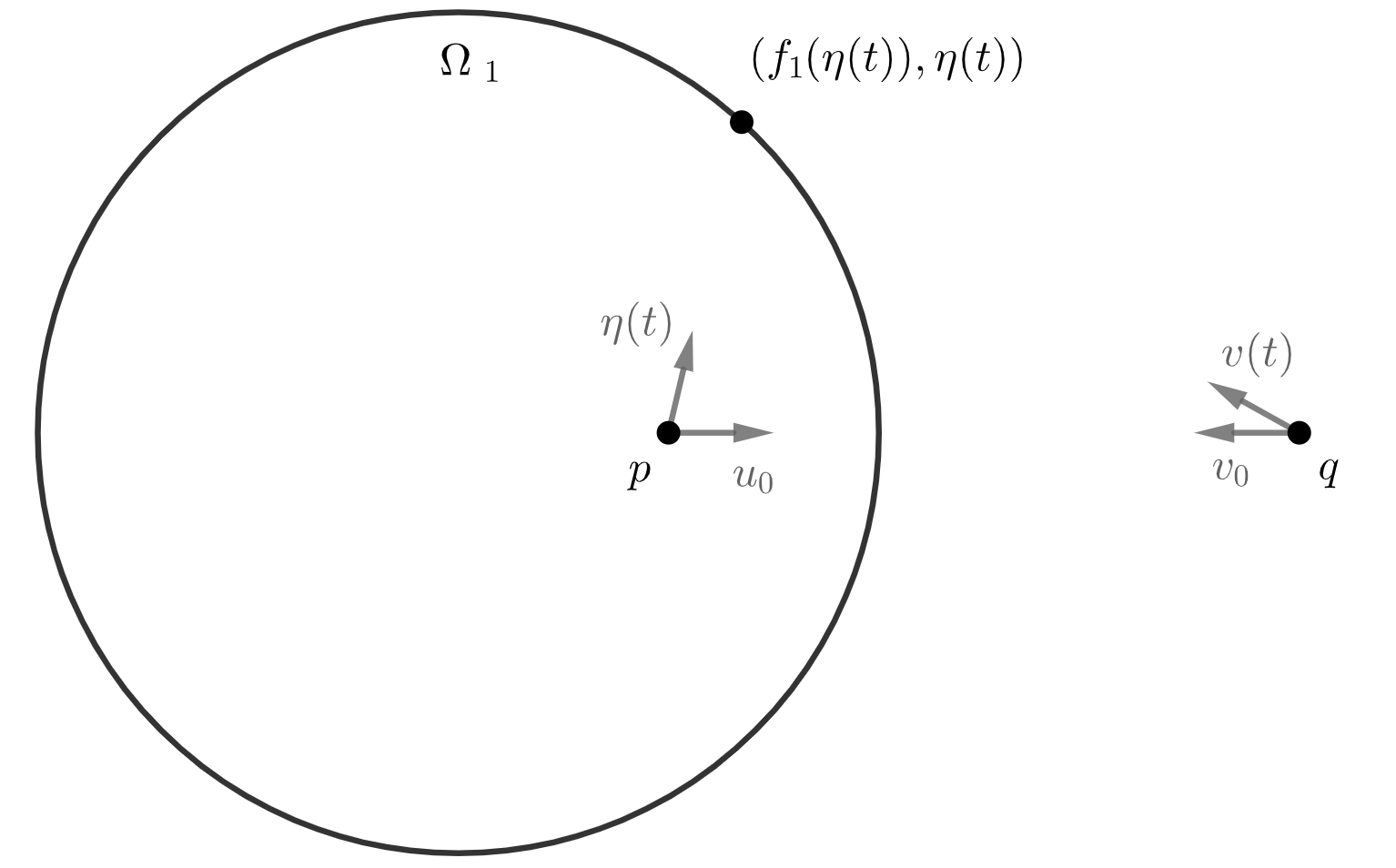}
    \caption{Proof of Lemma \ref{visibilitylma}.}
    \label{visibilityfig}
\end{figure}

\begin{proof}
    Identifying $M$ with the tangent space to $M$ at $p$, the set $\Omega_1$ is invariant under orthogonal transformations fixing the line through $p$ and $q$. By this symmetry, it suffices to consider a unit-speed great circle $\eta:[0,T]\to S_pM$ joining $-u_0$ to $u_0$ and to show that there exists some $t_0\in[0,T]$ such that  $(f_1(\eta(t)),\eta(t))$ is visible$_1$ exactly when $t\ge t_0$, and that $\angle(\eta(t_0),u_0)\le c(\la)<\pi$.
    
    For every $t\in[0,T]$, let $v(t)\in S_qM$ be the initial velocity of the unit-speed geodesic joining $q$ to $(f_1(\eta(t)),\eta(t))$. As $t$ varies, $v$ traces an arc in $S_qM$, which, by symmetry, is contained in a great circle in $S_qM$. If $v_0$ is the initial velocity of the unit-speed geodesic from $q$ to $p$, then $v(0)=v(T)=v_0$ (see Figure \ref{visibilityfig}).
    
    By Lemma \ref{omega0geolma}, the path $v$ may visit each point at most twice, so since $v$ lies on a great circle and its endpoints coincide, the function $\angle(v(t),v_0)$ attains a unique maximum, say at $t=t_0$, and every value except $v(t_0)$ must therefore be attained by $v$ exactly twice, once in $[0,t_0)$ and again in $(t_0,T]$. Since $\angle(\eta(t),u_0)$ is decreasing in $t$, it follows from Lemma \ref{omega0geolma} that  $(f(\eta(t),\eta(t))$ is visible$_1$ exactly when $t\ge t_0$.\\
    Set $\theta_0:=\eta(t_0)$ and $\alpha_0:=\angle(\eta(t_0),u_0)$. Let $\gamma:[0,S]\to M$ be a unit speed geodesic joining $q$ to $x_0:=(f_1(\theta_0),\theta_0)$. Since $\gamma$ intersects $\Omega_1$ only at $x_0$, it is tangent to $\partial\Omega_1$ at $x_0$, so by the definition of $\Omega_1$,
    $$0=\frac{d}{dt}\Big\vert_{t=S}\Big(\la d(\gamma(t),p)-d(\gamma(t),q)\Big)=\la \left<\dot\gamma(S),\partial/
    \partial_r\right>-1,$$
    whence $\alpha_1:=\angle(\dot\gamma(S),\partial/\partial_r)=\arccos(\la^{-1})$. Now, both $\alpha_0$ and $\alpha_1$ are angles in the triangle $\triangle p q x_0$, so since $M$ is nonpositively curved, $\alpha_0\le \pi-\arccos(\la^{-1})$. This finishes the proof.
\end{proof}

\begin{cor}\label{visibilitycor}
    Under the assumptions and notations of the previous lemma, a point $x=(f(\theta),\theta)\in\partial\Omega$ is visible if and only if  $\angle(\theta,u_0)\le a$, if and only if $x\in\partial\Omega_1$.
\end{cor}
\begin{proof}
    Suppose that $x$ is visible. Then $d_\Omega(x,q)=d(x,q)$, so by Lemma \ref{Omegaprop0}, $d(x,q)=\la d(x,p)$ whence $x\in\partial\Omega_1$. Since $x$ is visible and $\Omega_1\subseteq\Omega$, $x$ is visible$_1$, and therefore by Lemma \ref{visibilitylma}, $\angle(\theta,u_0)\le a$.
    
    Suppose that $\angle(\theta,u_0)\le a$, and let $x_1:=(f_1(\theta),\theta)\in\partial\Omega_1$. Then by Lemma \ref{visibilitylma}, $x_1$ is visible$_1$. We argue by induction that $x_1\in\partial\Omega_n$ and $x_1$ is visible$_n$ for all $n\in\N$.
    
    Indeed, suppose that $x_1\in\partial\Omega_n$ and $x_1$ is visible$_n$. Then $d_{\Omega_n}(x_1,q)=d(x_1,q)=\la d(x_1,p)$ so  $x\in\partial\Omega_{n+1}$. If $\gamma$ is the unit-speed geodesic from $q$ to $x_1$, then since $\gamma$ does not intersect $\Omega_n$ nor $\Omega_1$, $$d_{\Omega_n}(\gamma(t),q)=d(\gamma(t),q)\le \la d(\gamma(t),p)$$ so $\gamma(t)\notin\Omega_{n+1}$ for all $t$. Hence $x_1$ is visible$_{n+1}$. This completes the induction step. Thus $x_1\in\partial\Omega_n$ and $x_1$ is visible$_n$ for all $n\in\N$, whence $x_1\in\partial\Omega$. It follows that $x = x_1$, because $x$ and $x_1$ share the same $\theta$-coordinate and both lie on $\partial\Omega$. In particular, $x \in \partial\Omega_1$.
    
    Finally, if $x\in\partial\Omega_1$ then $d(x,q)=\la d(x,p)=d_\Omega(x,q)$ by Lemma \ref{Omegaprop0}, so $x$ is visible.
\end{proof}

\begin{proof}[Proof of Proposition \ref{unbounded}]
    Let $\theta \in S_pM$ satisfy $f(\theta) < \infty$, i.e. $$x: =(f(\theta),\theta)\in\partial \Omega.$$ We will show that
    \begin{equation}\label{sn-1distancegoal}
        \angle(\theta,u_0)<\pi,
    \end{equation} provided that $\ell$ is larger that some $L>0$ independent of $\theta$. This proves that $f(-u_0) = \infty$, whence $\Omega(=B_\infty)$ contains the ray $\{\theta=-u_0\}$ and is therefore unbounded.
    
    Let $\gamma:[0,T]\to M\setminus \Omega$ be a minimal path joining $q$ and $x$ in $M\setminus\Omega$, parametrized by speed $\la$, so that
    $$d_\Omega(\gamma(t),q)=\la t\qquad \text{ for all } t\in[0,T].$$
    We now prove that there is some $t_0>0$ such that $\gamma\vert_{[0,t_0]}$ is a geodesic, while $\gamma(t)\in\partial \Omega$ for all $t\ge t_0$. On each interval of $\gamma^{-1}(M\setminus\ol{\Omega})$, $\gamma$ must be a geodesic. But there cannot be such an interval both of whose endpoints lie in $\gamma^{-1}(\partial\Omega)$, because the function $$\la d(\gamma(t),p)-d_\Omega(\gamma(t),q)=\la d(\gamma(t),p)-\la t$$ is convex on such an interval (since $\gamma$ is a geodesic and $M$ is nonpositively curved), and by Lemma \ref{Omegaprop0} it vanishes on both endpoints of the interval and is positive in its interior, a contradiction. Therefore, there is some $t_0>0$ such that $\gamma\vert_{[0,t_0]}$ is a geodesic, while $\gamma(t)\in\partial \Omega$ for all $t\ge t_0$. By Corollary \ref{visibilitycor}, if we write in polar coordinates $$\gamma(t_0)=(r_0,\theta_0)=(f(\theta_0),\theta_0),$$ then there exists some $0<c_0<\pi$, depending only on $\la$, such that
    \begin{equation}\label{theta01dist}
        \angle(\theta_0,u_0)\le a<\pi-c_0.
    \end{equation} 
    Lemma \ref{Omegaprop0} implies that for $t\ge t_0$, the curve $\gamma$ satisfies
    $$d(\gamma(t),p)=\la^{-1}d_\Omega(\gamma(t),q)=t,$$
    and since $\gamma$ is speed-$\la$, it follows that in polar coordinates, it takes the form
    \begin{equation*}
        \gamma(t)=(r_0+(t-t_0),\eta(t)) \qquad t_0\le t\le T
    \end{equation*}
    where $\eta:[t_0,T]\to S_pM$ is a Lipschitz path satisfying
    \begin{equation*}
        |\dot\eta(t)|=\sqrt{\la^2-1}/G(r_0+(t-t_0))
    \end{equation*}
    (this derivative being defined almost everywhere). Therefore
    \begin{equation}\label{etadistbound}
        \angle(\eta(T),\eta(t_0))\le \int_{t_0}^T|\dot\eta(s)|ds=\sqrt{\la^2-1}\int_{t_0}^T\frac{1}{G(r_0+(s-t_0))}ds.
    \end{equation}
    Now, it follows from the triangle inequality and the definition \eqref{omega1} of $\Omega_1$ that 
    \begin{equation*}
        \mathcal{B}(p,\ell/(\la+1))\subseteq\Omega_1,
    \end{equation*}
    and therefore $r_0$ is bounded from below by $\ell/(\la+1)$. By the assumption $I<\infty$, if $\ell$ is large enough, the right hand side of \eqref{etadistbound} is strictly smaller than $c_0$, so by \eqref{theta01dist} and \eqref{etadistbound},
    \begin{equation*}
        \angle(\theta,u_0)\le \angle(\theta,\theta_0)+\angle(\theta_0,u_0)= \angle(\eta(T),\eta(t_0))+\angle(\theta_0,u_0)< \pi
    \end{equation*}
    as desired.
\end{proof}

In the case $\mathrm{dim}M=2$, the surface $M$ is either hyperbolic or parabolic (since it is rotationally-symmetric and non-compact). In the former case $I<\infty$, and in the latter $I=\infty$, see Milnor \cite{milnor}. Theorem \ref{mainthm} now follows.
\begin{rmk*}\normalfont
The proof of Proposition \ref{unbounded} enables us to identify the shape of $\Omega$ in the spherically symmetric, nonpositively curved case (both when $I=\infty$ and when $I<\infty$). Suppose for simplicity that $\mathrm{dim}M=2$, in which case, by passing to normal coordinates, we may assume that $M=(\R^2,g)$ for some rotationally symmetric metric $g$, and $p=(0,0)$. There exists $0\le a<\pi$ such that $(f(\theta)\cos\theta,f(\theta)\sin\theta)\in\partial\Omega$ is visible exactly when $-a\le \theta\le a$, and $\partial\Omega\cap\{|\theta|\le a\}$ coincides with $\partial\Omega_1$. The remainder $\partial\Omega\cap\{a\le|\theta|\le\pi\}$ is the union of two speed-$\la$ curves whose distance from the origin $p$ increases at rate $1$. In particular, if $M$ is the Euclidean plane, then  $a=\pi/2$, $\partial\Omega\cap\{x\ge0\}=\partial\Omega_1\cap\{x\ge0\}$ is a circular arc, and $\partial\Omega\cap\{x\le 0\}$ is the union of two logarithmic spirals (see Figure \ref{ABfig}). If $M$ is the unit disc with the Poincar{\'e} metric, then the shape $\Omega_1$ is not a circle, but the two curves are still logarithmic spirals.
\end{rmk*}

\end{document}